\documentclass[11pt]{amsart}
\usepackage{amssymb}
\usepackage{hyperref}
\usepackage{color}

\numberwithin{equation}{subsection}

\newtheorem{theorem}{Theorem}[section]
\newtheorem*{theorem*}{Theorem}

\newtheorem{proposition}[theorem]{Proposition}
\newtheorem{lemma}[theorem]{Lemma}

\theoremstyle{definition}

\newtheorem{remark}[theorem]{Remark}

\def\soc{{\rm soc}}
\def\Aut{{\rm Aut}}
\def\Sym{{\rm Sym}}
\def\GL{{\rm GL}}

\def\AGL{{\rm AGL}}
\def\AGAL{{\rm A\Gamma L}}
\def\GAL{{\rm \Gamma L}}

\newcommand{\C}{\mathrm{C}}
\newcommand{\Cent}{\mathbf{C}}
\newcommand{\E}{\mathrm{E}}

\newcommand{\V}{\mathrm{V}}
\newcommand{\Z}{\mathrm{Z}}
\newcommand{\FF}{\mathbb{F}}

\renewcommand{\leq}{\leqslant}
\renewcommand{\geq}{\geqslant}

\begin{document}
\title[$2$-arc-transitive graphs]{On $2$-arc-transitive graphs of order $kp^n$}

\author[L. Morgan, E. Swartz, G. Verret]{Luke Morgan, Eric Swartz, Gabriel Verret}

\address{Luke Morgan, Eric Swartz and Gabriel Verret, School of Mathematics and Statistics, 
\newline\indent University of Western Australia, 35 Stirling Highway, Crawley, WA 6009, Australia.} 
\email{luke.morgan@uwa.edu.au,eric.swartz@uwa.edu.au,gabriel.verret@uwa.edu.au}

\address{Gabriel Verret, FAMNIT, University of Primorska, Glagolja\v{s}ka 8, SI-6000 Koper, Slovenia.}

\thanks{The research of the first author is supported by the Australian Research Council grant
DP120100446. The research of the second author is supported by the ARC grant
DP120101336. The last author is supported by UWA as part of the ARC grant DE130101001.}


\begin{abstract}
We show that there exist functions $c$ and $g$ such that, if $k$, $n$ and $d$ are positive integers with $d> g(n)$ and $\Gamma$ is a $d$-valent $2$-arc-transitive graph of order $kp^n$  with $p$ a prime, then $p\leqslant kc(d)$. In other words, there are only finitely many $d$-valent 2-arc-transitive graphs of order $kp^n$ with $d>g(n)$  and $p$ prime. This generalises a recent result of Conder, Li and Poto\v{c}nik.
\end{abstract}

\maketitle

\section{Introduction}

All graphs considered in this paper are finite, connected and simple (they are  undirected and do not have loops or multiple edges). 
A \emph{$2$-arc} of a graph is a triple $(u,v,w)$ of  pairwise distinct vertices such that $v$ is adjacent to both $u$ and $w$. We say that a graph is \emph{$2$-arc-transitive} if its automorphism group acts transitively on its $2$-arcs.

The class of  2-arc-transitive graphs has attracted a lot of interest. Although many  partial classification results have been obtained,  a full classification might be out of reach. A nice survey of some of the main results in this area can be found in~\cite{seress}. Recently, Conder, Li and Poto\v{c}nik have proved the following.

\begin{theorem*}[{\cite[Theorem 1]{ConderLiPot}}]\label{ConderLiPotTheo}
Let $k$ be a positive integer.
\begin{enumerate}
\item If $d\geqslant 3$, then there exist only finitely many $d$-valent $2$-arc-transitive graphs of order $kp$  with $p$ a prime.
\item If $d\geqslant 4$, then there exist only finitely many $d$-valent $2$-arc-transitive graphs of order $kp^2$  with $p$ a prime.
\end{enumerate}
\end{theorem*}

Inspired by this result, we are naturally led to ask whether an analogous statement holds for graphs of order $kp^3$,  $kp^4$, etc. This is exactly the content of our main theorem:

\begin{theorem}\label{main}
There exist functions $c : \mathbb N \rightarrow \mathbb N$ and $g: \mathbb N \rightarrow \mathbb N$ such that, if $k$, $n$ and $d$ are positive integers with $d> g(n)$ and there exists a $d$-valent $2$-arc-transitive graph of order $kp^n$  with $p$ a prime, then $p\leqslant kc(d)$.
\end{theorem}

In other words, if $k$, $n$ and $d$ are fixed with $d$ large enough, then there are only finitely many $d$-valent $2$-arc-transitive graphs of order $kp^n$ with $p$ a prime. In some sense, this shows that, for classifying  $2$-arc-transitive graphs along these lines, the most interesting case is when the valency is small. Indeed, there has been much activity in classifying such graphs, especially with $n$ and $k$ small. (For example, an overview of the case $(n,d)=(1,3)$  can be found in~\cite[Section 6]{ConderLiPot}.)

The proof of Theorem~\ref{main}, which can be found in Section~\ref{sec:proof}, divides naturally into the affine and non-affine cases. Preparatory work for these  cases is done in Sections~\ref{sec:nonaffine} and \ref{sec:affine}, culminating in Theorems~\ref{theo:nonaffine} and \ref{theo:affine}. (In fact, Theorem~\ref{theo:nonaffine}   is stronger than required.)  To complete the proof of Theorem~\ref{main}, we also require a result of Trofimov and Weiss that depends upon the Classification of the Finite Simple Groups (CFSG)~\cite{GLS}. On the other hand, all of the results in Section~\ref{sec:pre} are CFSG-free.

\section{Preliminaries}
\label{sec:pre}

We begin with some preliminaries 
that set the stage for the proof of Theorem~\ref{main}. We denote the cyclic group of order $n$ by $\C_n$ and, for a prime power $d$, the elementary abelian group of order $d$ by $\E_d$. The \emph{soluble radical} of a group is its largest  normal soluble subgroup. 

We write $H \lesssim G$ if $H$ is isomorphic to a subgroup of $G$. We will say that a group $H$ is  \emph{involved} in a group $G$  if there are subgroups $K$ and $N$ of $G$ such that  $N$ is a normal subgroup of $K$ and $K/N \cong H$. 

A permutation group is called \emph{quasiprimitive} if each of its non-trivial normal subgroups is transitive. A transitive permutation group is called \emph{$2$-transitive} if a point-stabiliser is transitive on the remaining points. It is easy to see that a $2$-transitive group is quasiprimitive.

A graph is $G$-vertex-transitive if $G$ is a group of automorphisms of the graph acting transitively on its vertices. Let $\Gamma$ be a $G$-vertex-transitive graph and let $v$ be a vertex of $\Gamma$. We denote the set of neighbours of $v$  in $\Gamma$ by $\Gamma(v)$. We write $G_v^{\Gamma(v)}$ for the permutation group induced by the action of $G_v$ on $\Gamma(v)$ and  $G_v^{[1]}$ for the kernel of this action. Given a permutation group $L$, the pair $(\Gamma,G)$ is said to be \emph{locally-$L$}  if  $G_v^{\Gamma(v)}$ is permutation isomorphic to $L$. Note that a graph $\Gamma$ is 2-arc-transitive if and only if the pair $(\Gamma,\mathrm{Aut}(\Gamma))$ is locally-$L$ with $L$ a $2$-transitive group.

\subsection{Graph-restrictive groups and a key lemma}

Following \cite{junior}, we say that a transitive group $L$ is \emph{graph-restrictive} if there exists a constant $c(L)$ such that, for every locally-$L$ pair $(\Gamma,G)$ and $v$ a vertex of $\Gamma$, the inequality $|G_v|\leqslant c(L)$ holds. 

The following lemma, which is inspired by~\cite[Theorem~2]{ConderLiPot}, is the crucial first step in our proof of Theorem~\ref{main}.

\begin{lemma}   \label{mainlemma}
Let $L$ be a quasiprimitive graph-restrictive permutation group with corresponding constant $c(L)$, let $k$ and $n$ be positive integers and let $p$ be a prime with $p> kc(L)$.  If $(\Gamma,G)$ is a locally-$L$ pair such that $\Gamma$ has order $kp^n$  and $v$ is a vertex of $\Gamma$, then the following hold:
\begin{enumerate} 
\item   $|G_v|$ is coprime to $p$; 
\item $G_v$ is isomorphic to a subgroup of $\GL(n,p)$. 
\end{enumerate} 
\end{lemma}
\begin{proof}
Let $P$ be a Sylow $p$-subgroup of $G$. By vertex-transitivity we have $|G| = kp^n |G_v|$. Since $L$ is graph-restrictive, we have $k|G_v| \leqslant kc(L) < p$ hence $|P|=p^n$, $|G_v|$ is coprime to $p$ and  $P_v=1$. Moreover, since $|G:P| < p$, it follows from Sylow's Theorem that $P$ is normal in $G$.

Let $C$ be the centraliser of $P$ in $G$ and let $\Z(P)$ be the centre of $P$. Note that $C$ is normal in $G$ and $\Z(P) = P \cap C$. Since $P$ is normal in $G$, $\Z(P)$ is a Sylow $p$-subgroup of $C$ and the Schur-Zassenhaus Theorem \cite[6.2.1]{gorenstein} yields $C=\Z(P)\times J$ for some characteristic subgroup $J$ of $C$. Since $P_v=1$, it follows that $C_v=J_v$.

Suppose that $J_v\neq 1$. In particular, since $G_v^{\Gamma(v)}$ is quasiprimitive, $J$ has at most two orbits on the vertices of $\Gamma$. (See for example~\cite[Lemma~4]{ConderLiPot}.) Since $J$ is characteristic in $C$, it is normal in $G$, and thus these orbits have the same size. Since $p>2$, it follows that $p$ divides the size of these orbits, contradicting the fact that $|J|$ is coprime to $p$.

It follows that $C_v=J_v=1$, and thus $G_v$ is isomorphic to a subgroup $X$ of $\Aut(P)$. By \cite[5.3.5]{gorenstein}, we have that $X$ acts faithfully on the Frattini quotient $P/\Phi(P)$, which is elementary abelian of rank at most $n$. Thus $X$  is isomorphic to a subgroup of  $\mathrm{GL}(n,p)$.
\end{proof}

In view of Lemma~\ref{mainlemma}, we are  led to consider the following definition. For a finite group $X$, let
$$ \lambda(X) := \mathrm{min} \{n \mid X \text{ is involved in a finite subgroup of } \GL(n, \FF),~\FF \text{ a field with } \mathrm{char}(\FF)  \nmid |X| \}.  $$
 Note that, if $Y$ is involved in $X$, then $\lambda(Y)\leqslant\lambda(X)$. The next lemma shows that, when considering $\lambda(X)$, it suffices to work over the field of complex numbers.

\begin{lemma}\label{lemma:new}
If $X$ is a finite group, then there exists a 
finite subgroup  $G$ of $\GL(\lambda(X),\mathbb C)$ such that $X$ is involved in $G$ and every prime divisor of $|G|$ divides $|X|$.
\end{lemma}
\begin{proof}
By definition, there is a field $\FF$ with $\mathrm{char}(\FF) \nmid |X|$ and a finite   subgroup $G$ of $\GL(\lambda(X),\FF)$ such that, for some normal subgroup $K$ of $G$, we have $G/K \cong X$. Choose $G$ such that $|G|$ is minimal. Without loss of generality, we may assume that $\FF$ is algebraically closed.

We claim that $K$ is nilpotent.  (The argument used to prove this claim is taken from the proof of \cite[Lemma 5.5A(ii)]{dixon-mortimer}).  Let $p$ be a prime and let $P$ be a Sylow $p$-subgroup of $K$. The Frattini Argument yields $G=\mathrm N_G(P) K$ and hence $\mathrm N_G(P)/ \mathrm N_K(P) \cong N_G(P)K/K \cong G/K \cong X$. The minimality of $|G|$  implies that $G=\mathrm N_G(P)$. In particular, $P$ is normal in $K$ and hence $K$ is nilpotent.

We now show that every prime divisor of $|K|$ divides $|X|$. Let $p$ be a prime dividing $|K|$ and let $P$ be a Sylow $p$-subgroup of $K$. Since $K$ is nilpotent, we have $K=P \times Q$ for some characteristic subgroup $Q$ of $K$. If $p$ does not divide $|X|$, then the Schur-Zassenhaus Theorem \cite[6.2.1]{gorenstein} yields $G/Q \cong P \rtimes X$, and thus $G$ has a proper subgroup involving $X$, contradicting the minimality of $|G|$. 

In particular, every prime divisor of $|G|$ divides $|X|$ and thus $\mathrm{char}(\FF)$ does not divide $|G|$. Since $\FF$ is algebraically closed, this implies that the degrees of representations of $G$ over $\FF$ are the same as the degrees of  representations of $G$ over $\mathbb C$, see for instance \cite[Chapter 15]{Isaacs} and, in particular,  \cite[Theorem 15.13]{Isaacs}. We thus obtain a representation of $G$ over $\mathbb C$ of dimension $\lambda(X)$, completing the proof.
\end{proof}

\subsection{Locally non-affine pairs}\label{sec:nonaffine}
The finite quasiprimitive groups are classified (see \cite{praegerquasip}). If $G$ is such a group, then its socle has the form $T^\ell$ for some finite simple group $T$. If $T$ is abelian, then $G$ is called \emph{affine}; otherwise, we say that $G$ is \emph{non-affine}.  In this section, we consider the locally non-affine case of Theorem~\ref{main} (and, in fact, we prove a stronger result).

For $n\geqslant 1$, let 
\begin{equation}
\label{j defn}
J(n)=(n!)\cdot12^{n(\pi(n+1)+1)}
\end{equation}
 where $\pi(k)$ denotes the number of primes less than or equal to $k$.

\begin{lemma}\label{simple}
If $X$ is a finite group with trivial soluble radical, then $|X| \leqslant J(\lambda(X))$.
\end{lemma}
\begin{proof}
 By Lemma~\ref{lemma:new}, there exists a finite subgroup  $G$ of $\GL(\lambda(X),\mathbb C)$ such that $G$ has a normal subgroup $K$ with $G/K\cong X$. By a theorem of Jordan  \cite[Theorem 14.12]{Isaacs}, there exists an abelian normal subgroup $A$ of $G$ such that $|G:A| \leqslant J(\lambda(X))$. Since $G/K$ has trivial soluble radical, we have $A\leqslant K$ and hence $|X|=|G:K|\leqslant |G:A| \leqslant J(\lambda(X))$, as desired.
\end{proof}

\begin{theorem}\label{theo:nonaffine}
Let $k$, $n$ and $d$ be positive integers  with $d> J(n)$. Let $(\Gamma,G)$ be a locally-$L$ pair such that $\Gamma$ has order $kp^n$ for some prime $p$, $L$ has degree $d$ and is graph-restrictive and quasiprimitive. If $L$  is non-affine, then $p\leqslant kc(L)$.
\end{theorem}
\begin{proof}
We assume for a contradiction  that $p> kc(L)$. Let $v$ be a vertex of $\Gamma$. By Lemma~\ref{mainlemma}, $|G_v|$ is coprime to $p$ and $G_v$ is isomorphic to a subgroup of $\GL(n,p)$. In particular, $\lambda(L)\leqslant\lambda(G_v)\leqslant n$.

Let $S$ be the socle of $L$. Since $L$ is quasiprimitive, $S$ is transitive. Since $L$ is non-affine, $S$ is a direct product of non-abelian simple groups and thus has trivial soluble radical. Lemma~\ref{simple} then implies that $d \leqslant |S|\leqslant J(\lambda(S))\leqslant J(\lambda(L))\leqslant J(n)$, contradicting the fact that $d> J(n)$.
\end{proof}

\begin{remark}

It was conjectured by Praeger \cite{praegerconjecture} that finite quasiprimitive groups are graph-restrictive. The validity of this conjecture would render the graph-restrictive assumption in the hypothesis of Theorem~\ref{theo:nonaffine}  superfluous. The conjecture remains open but has been shown to hold in certain cases \cite{PSVRestrictive,pablo,trofweiss1,trofweiss2}. 
\end{remark}

\subsection{Locally affine pairs}\label{sec:affine}

In this section we consider locally-$L$ pairs where $L$ is a $2$-transitive   affine group. We first consider the case when $L$ is soluble. The  finite  soluble  $2$-transitive groups were classified by Huppert~\cite{Huppert}. A consequence of this classification is that, up to finitely many exceptions, all such groups are subgroups of the one-dimensional affine semilinear group, which we now define. 

Let $d=r^f$ be a power of a prime $r$. We denote the field of order $d$ by $\FF_d$  and the Galois group of the field extension $\FF_{r^f}/\FF_r$ by $\mathrm{Gal}(\FF_{r^f}/\FF_r)$. The group $\AGAL(1,d)$ is  $\langle T_{u,\alpha,\sigma} \mid u \in \FF_d$, $\alpha \in \FF_d^\#, \sigma \in  \mathrm{Gal}(\FF_{r^f}/\FF_r)  \rangle$ where $T_{u,\alpha,\sigma}$ is the permutation of $\FF_d$ defined by
\begin{equation}\nonumber
T_{u,\alpha,\sigma} :  x \mapsto \alpha(x^\sigma) + u, \hspace{0.5cm} x\in \FF_d.
\end{equation}
The permutation group $\AGAL(1,d)$ is $2$-transitive with a regular normal subgroup 
$$\V_d=\langle T_{u,1,1} \mid u \in \FF_d\rangle \cong \E_d$$ 
and point-stabiliser conjugate to 
$$\GAL(1,d):=\langle T_{0,\alpha,\sigma} \mid \alpha \in \FF_d^\#,\sigma \in   \mathrm{Gal}(\FF_{r^f}/\FF_r)     \rangle\cong \C_{d-1}\rtimes\C_f.$$
The point-stabiliser $\GAL(1,d)$ contains the normal subgroup 
$$\GL(1,d):=\langle T_{0,\alpha,1} \mid \alpha \in \FF_d^\#\rangle\cong \C_{d-1},$$
while 
$$\AGL(1,d):=  \langle T_{u,\alpha,1} \mid u \in \FF_d, \alpha \in \FF_d^\# \rangle =  \V_d \rtimes\GL(1,d)\cong \E_d\rtimes \C_{d-1},$$ 
 is a $2$-transitive normal subgroup of $\AGAL(1,d)$.

In the following omnibus proposition, we collect a few results concerning $2$-transitive subgroups of $\AGAL(1,d)$.

\begin{proposition}\label{BigProp}

Let $d=r^f$ be a power of a prime $r$, let $L$ be a $2$-transitive subgroup of $\AGAL(1,d)$ and let $X=L\cap \AGL(1,d)$.
The following hold:
\begin{enumerate}
\item $X=\V_d\rtimes X_0$; \label{newnew3}
\item $X_0$ is a subgroup of $\GL(1,d)$ of index at most $f$; \label{newnew}
\item every element of $\GAL(1,d)\setminus\GL(1,d)$ has order at most $\frac{d-1}{f}$ unless $r=2$ and $2\leq f\leq 6$; \label{newLemma}

\item every  element of $L_0$ has order at most $|X_0|$; \label{newnew2}
\item if $\FF$ is a field with $\mathrm{char}(\FF) \nmid |X|$ and $n$ is a positive integer such that $X \lesssim \GL(n,\FF)$ then  $n\geqslant |X_0|$. \label{AGL}
\end{enumerate} 
Moreover, if $(\Gamma,G)$ is a locally $L$-pair, $r$ and $f$ are coprime and  $v$ is a vertex of $\Gamma$ then
\begin{enumerate} \setcounter{enumi}{5}
 \item  $G_v$ contains a subgroup isomorphic to $X$. \label{agl subgroup}
\end{enumerate} 
\end{proposition}
\noindent \textit{Proof.}
We prove each claim in order.
\begin{enumerate}

\item The claim is clearly true if $d=4$ and thus we assume that $d\neq 4$. Since $L$ is soluble and $2$-transitive,  its socle $\soc(L)$ is elementary abelian and transitive, and thus regular with order $r^f$. We first show that $\soc(L)=\V_d$. Suppose for a contradiction that $\soc(L) \neq \V_d$. 

Note that $\soc(L)/(\soc(L) \cap V_d)\cong \soc(L) \V_d / \V_d\leq \AGAL(1,d)/\V_d\cong \GAL(1,d)$. Since the Sylow $r$-subgroups of $\GAL(1,d)$ are cyclic, we have $|\soc(L) \cap V_d | = r^{f-1}$. On the other hand, there exists  $vt \in \soc(L)$  with $v \in \V_d$ and $1\neq t\in \GAL(1,d)$. Since $vt$ has order $r$, $t$ must have order $r$ and therefore a conjugate $t'$ of  $t$ in $\GAL(1,d)$ lies in $\mathrm{Gal}(\FF_{r^f}/\FF_r)$. Now both $\V_d$ and $\soc(L)$ are abelian hence  $\soc(L) \cap \V_d$ is a subgroup of the centraliser $\Cent_{\V_d}(t)$ and we have $|\Cent_{\V_d}(t)|=|\Cent_{\V_d}(t')|$. By the Galois correspondence,  $\Cent_{\V_d}(t')$ is the subfield $\FF_{r^{f/r}}$ of $\FF_{r^f}$. This yields $r^{f-1} \leqslant r^{f/r}$, which is a contradiction since $d\neq 4$. We have shown that $\soc(L)=\V_d$. This implies that $\V_d\leq X$ and the result follows.

\item By~(\ref{newnew3}), $\V_d\leq L$ and thus $L_0\leq\GAL(1,d)$ and $X_0=L_0\cap \GL(1,d)$. Since $|\GAL(1,d):\GL(1,d)|=f$, we have $|L_0:X_0|\leq f$. Moreover, $L$ is $2$-transitive hence $|L_0|\geq d-1=|\GL(1,d)|$ and thus $|\GL(1,d):X_0|\leq f$.
\item 
Let $x$ be a generator of $\GL(1,d)$, let  $\sigma$ be a generator of $\mathrm{Gal}(\FF_d/\FF_r)$ and let $y=T_{0,1,\sigma}$. Now  $\GAL(1,d)=\langle x\rangle\rtimes\langle y\rangle$ where the action of $\langle y \rangle$  on $\langle x \rangle$ is the action of the Galois group of the field extension $\FF_{r^f}/\FF_r$ on $\FF_d^\#$.  We will show that any element of $\GAL(1,d)\setminus\langle x\rangle$ has order at  most $\frac{d-1}{f}$. Let $z$ be such an element and write $z=x'y'$ with $x' \in \langle x \rangle$ and $y' \in \langle y \rangle$. Note that $\langle z \rangle \cap \langle x \rangle$ is centralised by $z$ and $x'$ and thus by  $y'$. Let $e=|y'|$ and $k=\frac{f}{e}$. By the Galois  correspondence, the elements of $\mathbb F_d$ that are fixed by $y'$ are precisely those in the subfield $\mathbb F_{r^k}$.   It follows that $|\langle z \rangle \cap \langle x \rangle|$ divides $r^k-1$ and hence 
$$ |z | = |\langle z  \rangle :  \langle z \rangle \cap \langle x \rangle|| \langle z \rangle \cap \langle x \rangle| = |y'| | \langle z \rangle \cap \langle x \rangle| \leq e(r^{k}-1).$$
Since $f=ek$ and $e\geqslant 2$,  it is an easy exercise to show that $e(r^{k}-1)\leq \frac{r^f-1}{f}=\frac{d-1}{f}$ unless $r=2$ and $f\leq 6$.

\item By~(\ref{newnew}), $|X_0|\leq \frac{d-1}{f}$ and thus the claim follows by~(\ref{newLemma}) unless $r=2$ and $2\leq f\leq 6$. In the latter case, the claim can be checked by computer (for example, with the help of {\sc Magma}~\cite{Magma}).

\item Let $U$ be the natural $\GL(n,\FF)$-module considered as an $X$-module.  Note that $X$ is a Frobenius group with kernel $\V_d$ and complement $X_0$. Since the characteristic of $\FF$ does not divide $|X|$,  Maschke's Theorem~\cite[Theorem 1.9]{Isaacs} implies that $U$ is a completely reducible $X$-module. Moreover, since $\V_d$ acts non-trivially on $U$, we have  $U=\Cent_U(\V_d)\oplus W$, where $\Cent_U(\V_d)$ is the submodule of $U$ fixed by every element of $\V_d$ and $W$ is a non-zero submodule $W$ of $U$. Now $\Cent_W(\V_d)=0$ and  hence we may apply \cite[Theorem 15.16]{Isaacs}, which shows that the dimension of $W$ is divisible by $|X_0|$. The result follows.

\item Let $u$ be a neighbour of $v$ in $\Gamma$ and let $G_{uv}^{[1]}=G_u^{[1]}\cap G_v^{[1]}$. Note that $L\cong G_v/G_v^{[1]}$. In particular, if $G_v^{[1]}=1$, then the result is immediate. We therefore assume that $G_v^{[1]} \neq 1$ and thus $d\geq 3$. If $d=3$ then $\AGAL(1,d)=\AGL(1,d)\cong\Sym(3)$ and thus $L=X\cong\Sym(3)$ and the result follows from \cite{djokmiller}. 

Since $r$ and $f$ are coprime, we may thus assume that $d\geqslant 5$. By \cite[Theorem (ii)]{weissp}, we have $G_{uv}^{[1]}= 1$. In particular,  $G_v^{[1]}$ is isomorphic to a subgroup of $G_{uv}/G_u^{[1]}$, and the latter group is itself isomorphic to a subgroup of $\GAL(1,d)\cong \C_{d-1}\rtimes\C_f$. Now $f$ and $d-1$ are coprime to $r$, hence $|G_v^{[1]}|$ is coprime to $r$.  Let $R$ be a Sylow $r$-subgroup of $G_v$. Since the order of $G_v^{[1]}$ is coprime to $r$, we see that $R G_v^{[1]}/G_v^{[1]}$ is a Sylow $r$-subgroup of $G_v/G_v^{[1]}$. Thus $R G_v^{[1]}$ is a normal subgroup of $G_v$.  We claim that $R$ is normal in $G_v$.

Since $G_u^{[1]}$ and $G_v^{[1]}$ are normal subgroups of $G_{uv}$, it follows that $[G_v^{[1]},G_u^{[1]}] \leqslant G_{uv}^{[1]}=1$. Let $T$ be the normal closure in $G_v$ of $G_u^{[1]}$ and observe that $[T, G_v^{[1]} ] = 1$. Since $T$ is normal in $G_v$ and $T\nleqslant G_v^{[1]}$ (for otherwise $G_u^{[1]} \leqslant G_v^{[1]}$ and this yields $G_v^{[1]}=1$), by the quasiprimitivity of $L$, we have $R G_v^{[1]} \leqslant T G_v^{[1]}$. Since $|TG_v^{[1]}:T|$ divides $|G_v^{[1]}|$, which is coprime to $r$, $T$ contains a Sylow $r$-subgroup of $TG_v^{[1]}$. The normality of $T$ in $G_v$ implies $T$ contains every Sylow $r$-subgroup of $TG_v^{[1]}$. It follows that $R \leqslant T$ and thus $R$ centralises $G_v^{[1]}$. Now $R G_v^{[1]} = R\times G_v^{[1]}$, thus $R$ is characteristic in $R G_v^{[1]}$, and therefore $R$ is normal in $G_v$. Moreover $G_v = R \rtimes G_{uv}$ since $|G_{uv}|=|G_{uv}:G_v^{[1]}||G_v^{[1]}|$ is coprime to $r$.

Since $G_{uv}^{[1]}=1$ we see that $G_{uv}$ is isomorphic to a subgroup of $G_{uv}/G_v^{[1]} \times G_{uv}/G_u^{[1]}$ where 
$$G_{uv}/G_v^{[1]} \cong G_{uv}/G_u^{[1]}\cong L_0.$$
Note that $G_{uv}$ projects onto $L_{0}$ in both coordinates of the direct product.  Let $\pi: G_{uv} \mapsto L_0$ be the projection onto the first coordinate and let $g$ be an element of $G_{uv}$ of minimal order such that $\pi(g)$  generates $X_0$. Write $g=(x,g_2)$ with $g_2\in L_0$. By~(\ref{newnew2}), $g_2$ has order at most $|X_0|$ and thus $g$ has order $|X_0|$. It follows that $\langle R,g\rangle=R \rtimes\langle g\rangle\cong  \V_d\rtimes X_0=X$. \hfill \qed \end{enumerate}


\vspace{0.2cm}

To complete the case when $L$ is soluble, we will also need the following.

\begin{lemma}
\label{el ab sections of glnp}
If $r^f$ is a power of a prime $r$, then $\lambda(\E_{r^f})\geqslant \frac{2f}{3}$.
\end{lemma}
\begin{proof}
By Lemma~\ref{lemma:new},  $\E_{r^f}$ is involved in a finite $r$-subgroup  $R$ of $\GL(\lambda(\E_{r^f}),\mathbb C)$. In particular, there is an integer $f'$ such that $f\leqslant f'$ and $R/\Phi(R) \cong \E_{r^{f'}}$. It then follows by  \cite[Theorem~A]{isaacsrank} that $f \leqslant f' \leqslant \frac{3n}{2}$, as required.
\end{proof}

For the insoluble case, we prove the following result.

\begin{lemma}\label{lem:E(H)}
There exists an increasing function $I : \mathbb{N} \rightarrow \mathbb{N}$ such that, if $H$ is a finite insoluble affine $2$-transitive group  then $|H| \leqslant I(\lambda(H))$.     
\end{lemma}
\begin{proof}
Let $n \in \mathbb N$.  We will show that there is an upper bound on $|H|$ as $H$ runs over the finite insoluble affine $2$-transitive groups with $\lambda (H) \leqslant n$. This will allow us to define 
$$I(n) = \max\{|H| : H \textrm{ finite insoluble affine $2$-transitive, } \lambda(H) \leqslant n \}$$
with the required properties.

Let $H$ be such a group, let  $R(H)$ be the soluble radical of $H$ and let $T(H)$ be the socle of $H/{R(H)}$. By \cite[Theorem 6.1]{Hering}, $T(H)$ is a non-abelian simple group.  We have $\lambda(T(H))\leq\lambda(H)\leq n$, and it follows by Lemma~\ref{simple} that $|T(H)|\leqslant J(n)$. By~\cite[Corollary 6.3]{Hering}, for a given finite non-abelian simple group $T$, there are only finitely many finite  $2$-transitive groups $H$ with $T(H) \cong T$. This concludes the proof.
\end{proof}

For a positive integer $n$, let 
\begin{equation}\label{h defn}
h(n)=\max\{I(n), 23^2,(3n/2)^{3n/2}\}.
\end{equation}
It was shown in~\cite{weissp}  that affine $2$-transitive groups are graph-restrictive. Hence, in the hypothesis of the following theorem,  $c(L)$ is well-defined.

\begin{theorem}\label{theo:affine}
Let $k$, $n$ and $d$ be positive integers  with $d> h(n)$. Let $(\Gamma,G)$ be a locally-$L$ pair such that $\Gamma$ has order $kp^n$ for some prime $p$, $L$ has degree $d$ and is $2$-transitive. If $L$  is affine, then $p\leqslant kc(L)$.
\end{theorem}
\begin{proof}
Since $L$ is affine, $d$ is a prime power, say $d=r^f$ for some prime $r$. We assume for a contradiction  that $p> kc(L)$. Let $v$ be a vertex of $\Gamma$. By Lemma~\ref{mainlemma}, $|G_v|$ is coprime to $p$ and $G_v$ is isomorphic to a subgroup of $\GL(n,p)$. In particular, $\lambda(L)\leqslant\lambda(G_v)\leqslant n$.

We first assume that $L$ is soluble.  Since $d>23^2$,  it follows by \cite[XII, 7.3]{huppertiii} that $ L\leqslant \AGAL(1,d)$. Let $X= L\cap \AGL(1,d)$. 

If $r>f$ then, by Proposition~\ref{BigProp}(\ref{agl subgroup}),
 $G_v$ contains a subgroup isomorphic to $X$ and thus so does $\GL(n,p)$. By Proposition~\ref{BigProp}(\ref{AGL}), this implies $n\geqslant |X_0|$. Finally, Proposition~\ref{BigProp}(\ref{newnew}) yields $|X_0|\geqslant\frac{d-1}{f}$ and thus $n\geqslant \frac{d-1}{f}\geq\frac{d-1}{\log_2(d)}$, contradicting the fact that $d>\max\{23^2,(3n/2)^{3n/2}\}$. 

We may thus assume that $r\leqslant f$. Since the group $\E_d=\E_{r^f}$ is involved in $L$, it is involved in $\GL(n,p)$ and hence Lemma~\ref{el ab sections of glnp} gives $f\leqslant 3n/2$ and thus $d=r^f\leqslant f^f\leqslant (3n/2)^{3n/2}$, contradicting the fact that $d> h(n)$.

 We may thus assume that $L$ is insoluble. Lemma~\ref{lem:E(H)} implies that $d \leqslant |L|\leqslant I(\lambda(L)) \leqslant I(n)$, contradicting the fact that $d > h(n) \geqslant I(n)$. This final contradiction yields that $p\leqslant kc(L)$.  
\end{proof}

\section{Proof of Theorem~\ref{main}}\label{sec:proof}

By \cite[Theorem 1.4]{trofweiss1} $2$-transitive groups are graph-restrictive.  Hence we may define  $c(d)$ to be the maximum of $c(L)$ as $L$ runs over the $2$-transitive groups of degree $d$. Let  $J$ and $h$ be as in (\ref{j defn}) and (\ref{h defn}), respectively, and, for a positive integer $n$, let $g(n)=\max\{J(n),h(n)\}$. The proof now follows by applying Theorems~\ref{theo:nonaffine} and~\ref{theo:affine}. \hfill \qed
\begin{remark}
 By consulting the references, it is possible to explicitly compute the functions $c$ and $g$ defined above. 
 Although one can find better bounds than the ones given, we choose not to attempt to optimise these functions, being satisfied merely with their existence.
\end{remark}


\noindent\textsc{Acknowledgements.}
We are grateful to Michael Giudici for pointing out a mistake in an earlier version of this paper.

\end{document}